\documentclass[alphabetic, 
        colorlinks,
        linkcolor=red,  citecolor=blue,
        backref
]{amsart}

\usepackage{enumerate, longtable}
\usepackage{amsmath, amscd, amsfonts, amsthm, amssymb, latexsym, comment, stmaryrd, graphicx, dsfont, amsaddr, mathtools, stmaryrd, orcidlink}
\usepackage{xcolor}
\usepackage[all]{xy}
\usepackage{fullpage}
\usepackage{tikz}

\usepackage{amstext} 
\usepackage{array}
\usepackage{cleveref}

\numberwithin{equation}{section}

\newtheorem{claim}{Claim}[section]
 
\newtheorem{thm}{Theorem}

\newtheorem{prop}[claim]{Proposition}

\newtheorem{lem}[claim]{Lemma}

\theoremstyle{definition}

\newtheorem{definition}[claim]{Definition} 
\newtheorem{rem}[claim]{Remark}
 
\newtheorem{conv}[claim]{Convention}
\newtheorem{notation}[claim]{Notation}

\newcommand{\N}{{\mathbb N}}
\newcommand{\Z}{{\mathbb Z}}
\newcommand{\R}{{\mathbb R}}

\newcommand{\USp}{{\mathrm{USp}}}
\newcommand{\U}{{\mathrm{U}}}
\renewcommand{\O}{{\mathrm{O}}}
\newcommand{\SO}{{\mathrm{SO}}}
\newcommand{\tr}{{\mathrm{tr}}}
\newcommand{\F}{\mathbb{F}}

\definecolor{dgreen}{RGB}{0,128,0}
\definecolor{auburn}{rgb}{0.43, 0.21, 0.1}

\allowdisplaybreaks

\author{Alexei Entin }
\address{Raymond and Beverly Sackler School of Mathematical Sciences, Tel Aviv University, Tel Aviv 69978, Israel}
\email{aentin@tauex.tau.ac.il (corresponding author, ORCID: \orcidlink{https://orcid.org/0000-0002-4169-7102})}

\author{Noam Pirani}
\address{Raymond and Beverly Sackler School of Mathematical Sciences, Tel Aviv University, Tel Aviv 69978, Israel}
\email{noampirani@mail.tau.ac.il}


\title{Moments of traces of random symplectic matrices and hyperelliptic $L$-functions}
\date{}

\begin{document}

\begin{abstract}
    We study matrix integrals of the form
    $$\int_{\USp(2n)}\prod_{j=1}^k\tr(U^j)^{a_j}\mathrm d U,$$
    where $a_1,\ldots,a_r$ are natural numbers and integration is with respect to the Haar probability measure. We obtain a compact formula (the number of terms depends only on $\sum a_j$ and not on $n,k$) for the above integral in the non-Gaussian range $\sum_{j=1}^kja_j\le 4n+1$. This extends results of Diaconis-Shahshahani and Hughes-Rudnick who obtained a formula for the integral valid in the (Gaussian) range $\sum_{j=1}^kja_j\le n$ and $\sum_{j=1}^kja_j\le 2n+1$ respectively. We derive our formula using the connection between random symplectic matrices and hyperelliptic $L$-functions over finite fields, given by an equidistribution result of Katz and Sarnak, and an evaluation of a certain multiple character sum over the function field $\F_q(x)$.

    We apply our formula to study the linear statistics of eigenvalues of random unitary symplectic matrices in a narrow bandwidth sampling regime.
\end{abstract}

\maketitle

\section{Introduction}

An important tool in the study of random matrices drawn from the classical compact groups $G=\U(n),\USp(2n),\SO(n)$ are \emph{trace moments}, which are integrals of the form
\begin{equation}\label{eq: main int}\int_G\prod_{j=1}^k\left(\tr(U^j)^{a_j}\tr(U^{-j})^{b_j}\right)\mathrm dU,\quad a_j,b_j\in\Z_{\ge 0},\end{equation} where integration is with respect to the Haar probability measure on $G$. The study of such integrals was initiated by Dyson \cite{Dys62}, who evaluated (\ref{eq: main int}) in the case $G=\U(n),\,\sum_j(a_j+b_j)\le 2$, namely he showed that
$$\int_{\U(n)}\tr(U^j)\mathrm dU=0,\quad\int_{\U(n)}\tr(U^i)\tr(U^{-j})\mathrm dU=\delta_{ij}\min(j,n),\quad i,j\in\N$$
($\delta_{ij}$ is the Kronecker delta). He used these identities to compute the pair correlation of eigenvalues of random unitary matrices. It was later conjectured by Montgomery \cite{Mon73} that the zeros of the Riemann zeta function have the same pair correlation, an observation which was an early cornerstone in the study of the connections between $L$-functions and random matrices. For a survey of these connections see \cite{KaSa99a}.

Diaconis and Shahshahani \cite{DiSh94} (see also \cite{DiEv01} or \cite[\S 3.3]{Mec19} for a slightly corrected exposition) showed that
\begin{equation}\label{eq: DS unitary}\int_{\U(n)}\prod_{j=1}^k\left(\tr(U^j)^{a_j}\tr(U^{-j})^{b_j}\right)\mathrm dU=\prod_{j=1}^k\delta_{a_jb_j}j^{a_j}a_j!=\mathbb E\left(\prod_{i=1}^k(j^{1/2}Z_j)^{a_j}(\overline{j^{1/2}Z_j})^{b_j}\right),\end{equation} where $Z_1,\ldots,Z_k$ are independent standard complex Gaussians and $\sum_{j=1}^kj(a_j+b_j)\le 2n$. The relation (\ref{eq: DS unitary}) breaks down beyond this range, a phenomenon described in \cite{HuRu03} as the mock Gaussian behavior of $\tr(U^j)$. Diaconis and Shahshahani also derived corresponding formulae for the groups $G=\USp(2n),\,\SO(n)$ \footnote{Diaconis and Shahshahani computed trace moments for $\O(n)$ instead of $\SO(n)$. It is easy to deduce the same formula for $\SO(n)$ from the corresponding formulae for $\O(n),\USp(2m)$ using the Weyl integration formulae \cite[Theorem 3.5]{Mec19}.}. 
Before describing these we introduce some convenient notation and terminology.

Let $\mathcal P$ be the set of sequences of nonnegative integers with finite support. We denote elements of $\mathcal P$ by $\mathbf a = (a_1,a_2,\ldots),\,\mathbf b = (b_1,b_2,\ldots)\in\mathcal P$ etc. We define the length and the size of $\mathbf a\in\mathcal P$ to be
$$\ell(\mathbf a)=\sum_{j=1}^\infty a_j,\quad |\mathbf a|=\sum_{j=1}^\infty ja_j.$$
An element $\mathbf a\in\mathcal P$ can be viewed as a partition of the number $|\mathbf a|$ with $a_j$ parts of size $j$ (the total number of parts being $\ell(\mathbf a)$), and henceforth we will refer to elements of $\mathcal P$ as partitions (in the standard multiplicative notation for partitions $\mathbf a$ corresponds to $1^{a_1}2^{a_2}\cdots k^{a_k}$ where $k=\max\{j:a_j> 0\}$). Alternatively, an element $\mathbf a\in\mathcal P$ can be viewed as a finite multiset of natural numbers with each element $j$ having multiplicity $a_j$ (the cardinality of this multiset is $\ell(\mathbf a)$).
The set $\mathcal P$ is partially ordered by $$\mathbf b\le\mathbf a\quad\mbox{iff}\quad \forall j\in\mathbb N:\, b_j\le a_j.$$ If $\mathbf b\le\mathbf a$ we denote
$$\mathbf a-\mathbf b=(a_1-b_1,a_2-b_2,\ldots)\in\mathcal P.$$ 
We also define
$${\mathbf a\choose\mathbf b}=\prod_{j=1}^\infty{a_j\choose b_j},$$ 

\begin{equation}\label{eq: def g}g(\mathbf a)=\prod_{j=1}^\infty g_j(a_j),\quad
g_j(a)=\left[\begin{array}{ll}0,&2\nmid ja,\\
j^{a/2}(a-1)!!,&2\nmid j,\,2|a,\\
\sum_{l=0}^\infty {a\choose 2l}j^l(2l-1)!!,&2|j,\end{array}\right.\end{equation}
where $(2m)!!=2\cdot 4\cdots 2m,\,(2m-1)!!=1\cdot 3\cdots (2m-1)$ is the double factorial (in particular $0!!=(-1)!!=1$). Note that all the sums and products appearing above are actually finite. The significance of $g(\mathbf a)$ is that it expresses certain Gaussian moments, see (\ref{eq: DS orthogonal}), (\ref{eq: DS symplectic}) below, where the second equality in each line holds without any restrictions on $\mathbf a$.

If $G=\USp(2n), \SO(n)$ and $U\in G$ then $\tr(U^j)\in\R$ and $\tr(U^{-j})=\tr(U^j)$, so it is enough to study integrals of the form 
$$M(G,\mathbf a)=\int_{G}\prod_{j=1}^\infty \tr(U^j)^{a_j}\mathrm d U,\quad\mathbf a=(a_1,a_2,\ldots)\in\mathcal P$$
(note that the product is actually finite).

Diaconis and Shahshahani \cite{DiSh94}
(see also \cite{DiEv01} or \cite[\S 3.3]{Mec19} for a slightly corrected exposition) proved that

\begin{alignat}{5}\label{eq: DS symplectic} M(\USp(2n),\mathbf a)&=(-1)^{\ell(\mathbf a)}g(\mathbf a)&=\mathbb E\left(\prod_{j=1}^\infty(j^{1/2}X_j-\eta_j)^{a_j}\right),&\quad\quad&|\mathbf a|\le n,\\ \label{eq: DS orthogonal}
M(\SO(n),\mathbf a)&=\quad\quad g(\mathbf a)&=\mathbb E\left(\prod_{j=1}^\infty (j^{1/2}X_j+\eta_j)^{a_j}\right),&\quad\quad&|\mathbf a|\le \frac n2,\end{alignat}
where
\begin{equation}\label{eq: def eta}\eta_j=\left[\begin{array}{ll}1,&2\,|\,j,\\0,&2\nmid j,\end{array}\right.\end{equation} and $X_1,X_2\ldots$ are independent standard real Gaussians. 
Hughes and Rudnick \cite{HuRu03} showed that (\ref{eq: DS symplectic}) and (\ref{eq: DS orthogonal}) hold in the extended ranges $|\mathbf a|\le 2n+1$ and $|\mathbf a|\le n-1$ respectively.
Their method (based on cumulant expansions) is completely different from the representation-theoretic method of Diaconis-Shahshahani. Stolz \cite{Sto05} gave an alternative derivation of (\ref{eq: DS symplectic}),(\ref{eq: DS orthogonal}) in the extended ranges of Hughes-Rudnick using the original approach of Diaconis-Shahshahani. 
Pastur and Vasilchuk \cite{PaVa04} gave alternative proofs of (\ref{eq: DS unitary}),(\ref{eq: DS symplectic}),(\ref{eq: DS orthogonal}) using a new method inspired by statistical mechanics.

Using the methods of Hughes-Rudnick, Keating and Odgers \cite[Lemma 2]{KeOd08} derived formulae for $\int_G\tr(U^{j_1})\tr(U^{j_2})\mathrm dU,$ $G=\USp(2n),\SO(2n)$ valid for all $j_1,j_2$. 
These do not generally agree with (\ref{eq: DS symplectic}),(\ref{eq: DS orthogonal}) beyond the ranges established by Hughes and Rudnick, so the corresponding trace moments are not Gaussian beyond these ranges.

The main result of the present work is a formula for $M(\USp(2n),\mathbf a)$ in the extended non-Gaussian range $|\mathbf a|\le 4n+1$.

\begin{thm}\label{thm: main} Assume $|\mathbf a|\le 4n+1$. Then
\begin{equation}\label{eq: thm1}M(\USp(2n),\mathbf a)=(-1)^{\ell(\mathbf a)}\sum_{\mathbf b\le\mathbf a}{\mathbf a\choose\mathbf b}g(\mathbf b)\phi(n,\mathbf a-\mathbf b),\end{equation}
where 
\begin{equation}\label{eq: def phi}\phi(n,\mathbf c)=\left[\begin{array}{ll}
-\sum_{\mathbf d\le\mathbf c\atop{|\mathbf d|\le{|\mathbf c|}/2-n-1}}(-1)^{\ell(\mathbf d)}{\mathbf c\choose\mathbf d},&|\mathbf c|>0,\,2\mid |\mathbf c|,\\1,&|\mathbf c|=0, \\0,&|\mathbf c|>0,\,2\nmid|\mathbf c|. \end{array}\right.\end{equation}
\end{thm}

Note that if $|\mathbf a|\le 2n+1$ then the only nonzero term in
(\ref{eq: thm1}) is with $\mathbf b=\mathbf a$ and (\ref{eq: thm1}) reduces to (\ref{eq: DS symplectic}) in this case. Beyond this range the non-Gaussian corrections $\phi(n,\mathbf a-\mathbf b)$ begin to appear. The cutoff $4n+1$ emerges from the proof (see the proof of Proposition \ref{prop_S_sum_estimate} below). The identity (\ref{eq: thm1}) may hold beyond this range, but our method cannot treat larger values of $|\mathbf a|$.

Our method of proving Theorem \ref{thm: main} is completely different from the combinatorial and representation-theoretic approaches underlying previous results on trace moments for the classical compact groups. Instead we use the deep connection between random symplectic matrices and hyperelliptic $L$-functions over finite fields. Using an equidistribution result of Katz and Sarnak we relate the trace moments $M(\USp(2n),\mathbf a)$ to certain multiple character sums over primes in $\F_q[x]$ for large $q$. We are able to precisely estimate these sums provided $|\mathbf a|\le 4n+1$, which allows us to derive (\ref{eq: thm1}). The non-Gaussian correction terms $\phi(n,\mathbf a-\mathbf b),\,\mathbf b<\mathbf a$ arise from quadratic reciprocity and the functional equation for hyperelliptic $L$-functions over $\F_q$.

Previously the first author \cite{Ent14} gave a new derivation of (\ref{eq: DS unitary}) based on a similar approach, but using Artin-Schreier instead of hyperelliptic $L$-functions. The same approach was utilized by the first author, Roditty-Gershon and Rudnick \cite{ERR13} to compute the $n$-level density of the eigenvalues of random symplectic matrices for an improved range of test functions (a combinatorial calculation of this density, as well as a generalization to the orthogonal group, was later given by Mason and Snaith \cite{MaSn17}). The strategy of proving a complicated combinatorial identity with arithmetic significance by first proving the underlying arithmetic results over function fields also appears in the spectacular proof of the Langlands fundamental lemma by Ng\^o \cite{Ngo10}.

\begin{rem} It is natural to try to obtain an analog of Theorem 1 for the orthogonal group $\SO(n)$ using the family of Hasse-Weil $L$-functions of quadratic twists of a fixed elliptic curve $E/\F_q(x)$, which is known to have orthogonal symmetry. Unfortunately, it is hard to estimate the corresponding character sums directly beyond the Gaussian range $|\mathbf a|<n$.\end{rem}

\begin{rem} In principle it should be possible to derive Theorem \ref{thm: main} and its analog for $SO(n)$ using the approach of Hughes-Rudnick, but the calculations get significantly more technical in the extended range. The Diaconis-Shahshahani method appears to be even less technically convenient for extending the range, because the resulting combinatorial expressions become quite unwieldy. We chose the arithmetic approach because we had good reasons to expect in advance that the calculations would work out, and the formula (\ref{eq: thm1}) emerges in a transparent way from the arithmetic.
\end{rem}

\subsection{Applications to linear statistics of eigenvalues}

\begin{notation}
    $e(t)=e^{2\pi it}$.
\end{notation}

Integrals of the form (\ref{eq: main int}) are useful for the study of linear statistics of eigenvalues of random matrices drawn from the classical compact ensembles. These are expressions of the form
$W_F(U)=\sum_{i=k}^{N}F(\theta_k)$, where $e(\theta_1),\ldots,e(\theta_N)$ are the eigenvalues of $U$, the matrix $U$ is drawn at random from one of the groups $\U(n),\USp(2n),\SO(n)$ and $F\in C^{\infty}(\R/\Z)$ is a periodic test function. One is interested in the statistical behavior of $W_F(U)$, e.g. its moments, for large values of $n$. The test function $F$ may depend on $n$ and various regimes have been considered. Diaconis and Evans \cite{DiEv01} studied the regime of fixed $F$ and showed that the moments of $W_F(U)$ become Gaussian in the limit $n\to\infty$. Soshnikov \cite{Sos00} generalized this to a mesoscopic sampling regime. Hughes and Rudnick \cite{HuRu03, HuRu03a} studied the local sampling regime, that is when $F(t)=\sum_{u\in\Z}f(N(t+u))$ and $f\in\mathcal S(\R)$ is a fixed Schwartz function. In this regime the $m$-th moment of $W_F(U)$ is generally Gaussian only if $\hat f$ is supported on $[-1/m,1/m]$ (resp. $[-2/m,2/m]$) in the symplectic and orthogonal case (resp. unitary case). This is closely related to the breakdown of (\ref{eq: DS unitary}),(\ref{eq: DS symplectic}),(\ref{eq: DS orthogonal}) beyond the corresponding Gaussian ranges.

Rudnick and Waxman \cite[\S 5]{RuWa19} studied a narrow bandwidth sampling regime, which we now describe. We focus on the case $G=\USp(2n)$. Let $f\in C^{\infty}(\R/\Z)$ be a fixed even and real-valued test function. For a matrix $U\in\USp(2n)$ and $\nu\in\N$ define
$$W_{f,\nu}(U)=\sum_{k=1}^{2n}f(\theta_k)e(\nu\theta_k),$$
where $e(\theta_1),\ldots,e(\theta_{2n})$ are the eigenvalues of $U$. The Fourier coefficients of the function $f(t)e(\nu t)$ are concentrated around $\nu$ as $\nu$ becomes large. Note that $W_{f,\nu}(U)\in\R$ because the eigenvalues of $U$ come in conjugate pairs and $f$ is even and real-valued. Rudnick and Waxman \cite[Proposition 5.3]{RuWa19} computed the variance of $W_{f,\nu}(U)$ in the limit $n\to\infty$, where $f$ is fixed and $\nu$ grows linearly with $n$. Here we apply Theorem \ref{thm: main} to compute higher moments of $W_{f,\nu}(U)$ in a restricted range of $\nu$.

\begin{thm}\label{thm: waxman} Let $f\in C^{\infty}(\R/\Z)$ be a fixed even real-valued function, $m$ a fixed natural number, $\epsilon>0$ a fixed constant. Assume $\epsilon n\le \nu\le \frac{4n}m-n^{1/2}$. Then
$$\int_{\USp(2n)}W_{f,\nu}(U)^m\mathrm dU=\eta_m\cdot(m-1)!!\cdot\|f\|_{L^2}^{m}\nu^{m/2}+O_{f,m,\epsilon}\left(n^{(m-1)/2}\right),$$
where
$\|f\|_{L^2}=\left(\int_0^1f(t)^2\mathrm dt\right)^{1/2}$, $\eta_m$ is defined by (\ref{eq: def eta}) and the implicit constant may depend on $f,m,\epsilon$. Note that the main term is the $m$-th moment of the real Gaussian with mean 0 and variance $\|f\|_{L^2}^2\nu$.
\end{thm}

The derivation of Theorem \ref{thm: waxman} from Theorem \ref{thm: main} will be given in \S\ref{sec: waxman}.

\begin{rem}
    The assertion of Theorem \ref{thm: waxman} does not hold beyond the range $\nu<4n/m$. For example if $m=2$ the correct main term is $\|f\|_{L^2}^2\min(\nu,2n)$ by \cite[Proposition 5.3]{RuWa19}. It is an interesting question whether the moments remain (approximately) Gaussian with variance $\|f\|_{L^2}^2\min(\nu,2n)$ beyond the range $\nu<4n/m$.
\end{rem}

\begin{rem} Using \cite[Theorem 3(ii)]{HuRu03} in place of Theorem \ref{thm: main} one can derive a version of Theorem \ref{thm: waxman} with $\SO(2n)$ in place of $\USp(2n)$ (the RHS remains the same) in the narrower range $\epsilon n\le \nu\le\frac{2n}m-n^{1/2}$. Versions for $\SO(2n+1)$ and $U(n)$ can also be obtained in suitable ranges.\end{rem}

\section{Hyperelliptic $L$-functions and equidistribution}
\label{sec: hyperelliptic}

Let $q$ be an odd prime power. We will work in the ring of polynomials $\F_q[x]$, which has many parallels with $\Z$. For an introduction to number theory in $\F_q[x]$ see \cite{Ros02}. We will use $P, P_i, P_{ji}$ etc. to denote monic primes of $\F_q[x]$ and notation such as $\sum_P$ will always denote summation over monic primes.
Similarly, we will use $Q,Q_i,Q_{ji}$ etc. to denote monic prime powers in $\F_q[x]$, and sums such as $\sum_Q$ will always denote summation over monic prime powers. Moreover, $P_{ji}$ (resp. $Q_{ji}$) will always denote a prime (resp. prime power) of degree $j$. We denote by $\left(\frac AB\right)$ the Jacobi symbol and by
$$\Lambda(Q)=\left[\begin{array}{ll}\deg P,&Q=P^e,\\0,&Q\mbox{ not a prime power}\end{array}\right.$$
the von Mangoldt function. 

For $n\in\N$ let
\begin{equation}\label{eq: def H}\mathcal H_{2n+1}=\{h\in\F_q[x]\mbox{ monic squarefree},\,\deg h=2n+1\}.\end{equation} We have (\cite[Proposition 2.3]{Ros02})
\begin{equation}\label{eq: num sqfree}|\mathcal H_{2n+1}|=q^{2n+1}(1-q^{-1}).\end{equation}

A polynomial $h\in\mathcal H_{2n+1}$ defines a hyperelliptic curve $C_h$ with affine model given by $y^2=h(x)$. With it one associates an $L$-function
\begin{equation}\label{eq: def L}L(z,C_h)=\sum_{j=1}^\infty\sum_{F\,\mathrm{monic}\atop{\deg F=j}}\left(\frac hF\right)z^j=\prod_P\left(1-\left(\frac hP\right)z\right)^{-1},\end{equation}
which is a polynomial with integer coefficients of degree $2n$. Moreover by the Riemann Hypothesis for curves over finite fields
\begin{equation}\label{eq: L roots}L(z,C_h)=\prod_{i=1}^{2n}(1-\alpha_{h,i}z),\quad |\alpha_{h,i}|=q^{1/2}.\end{equation}
Taking $\frac{\mathrm d}{\mathrm dz}\log$ of the RHS of (\ref{eq: def L}) and (\ref{eq: L roots}) and comparing coefficients at $z^j$ gives the \emph{Weil explicit formula}
\begin{equation}\label{eq: weil exp}\sum_{i=1}^{2n}\alpha_{h,i}^j=-\sum_{\deg Q=j}\Lambda(Q)\left(\frac hQ\right).\end{equation}

One also associates with $C_h$ the \emph{Frobenius class} $\Theta_h$, which is a conjugacy class in $\USp(2n)$ with eigenvalues $q^{-1/2}\alpha_{h,1},\ldots,q^{-1/2}\alpha_{h,2n}$ (see \cite[9.2.4, 10.1.18.3, 9.1.16]{KaSa99}). 
A deep equidistribution theorem due to Katz and Sarnak says that for fixed $n$ and $q\to\infty$, the classes $\Theta_h,\,h\in\mathcal H_{2n+1}$ become equidistributed in the space of conjugacy classes of $\USp(2n)$.

\begin{thm}[Katz-Sarnak]\label{thm: katz-sarnak} Let $f$ be a continuous class function on $\USp(2n)$. Then $$\frac 1{|\mathcal H_{2n+1}|}\sum_{h\in\mathcal H_{2n+1}}f(\Theta_h)=\int_{\USp(2n)}f(U)\mathrm d U+O_n(q^{-1/2})$$
(here the implicit constant depends on $n$).
\end{thm}

\begin{proof} Combine \cite[Theorem 9.6.10]{KaSa99} with \cite[Theorem 10.1.18.3]{KaSa99}.\end{proof}

\begin{conv} Throughout \S\ref{sec: hyperelliptic}-\ref{sec: final}, $n\in\N,\,\mathbf a\in\mathcal P$ are fixed and all asymptotic notation has implicit constants which may depend on $n,\mathbf a$ and possibly other parameters explicitly fixed, but not on $q$.\end{conv}

\begin{prop}\label{prop_explicit_formula_means_of_traces} Let $\mathbf a$ be supported on $[1,k]$ (that is $a_j=0$ for $j>k$). Then \begin{multline*}(-1)^{\ell(\mathbf a)}q^{-2n-1-|\mathbf a|/2}\sum_{h\in\mathcal H_{2n+1}}\sum_{(Q_{ji}:\,1\le j\le k,\,1\le i\le a_j)\atop{\deg Q_{ji}=j}}\prod_{1\le j\le k\atop{1\le i\le a_j}}\Lambda(Q_{ji})\left(\frac h{Q_{ji}}\right)=\int_{\USp(2n)}\prod_{j=1}^\infty\tr(U^j)^{a_j}\mathrm d U+O(q^{-1/2}).\end{multline*}
Here (and elsewhere where similar notation appears) the internal sum on the LHS is over all choices of (monic) prime powers $Q_{ji},\,1\le j\le k,1\le i\le a_j$ with $\deg Q_{ji}=j$.
\end{prop}

\begin{proof}
    Let $\alpha_{h,1},\ldots,\alpha_{h,2n}$ be the inverse zeros of $L(z,C_h)$. Recall that $q^{-1/2}\alpha_{h,i}$ are the eigenvalues of $\Theta_h$. By the explicit formula (\ref{eq: weil exp}) we have 

    \begin{equation}\label{eqn_weil_explicit_formula}
    q^{j/2}\tr(\Theta_h^j)=\sum_{i=1}^{2n}\alpha_{h,i}^j=-\sum_{\deg Q=j}\Lambda(Q)\left(\frac hQ\right)   
    \end{equation}
     for any $j\in\N$. Raising (\ref{eqn_weil_explicit_formula}) to the $a_j$-th power, taking the product over $j$ and expanding the RHS, we obtain (recalling that $\mathbf a$ is supported on $[1,k]$)
    $$\prod_{j=1}^\infty\tr(\Theta_h^j)^{a_j}=\prod_{j=1}^k\tr(\Theta_h^j)^{a_j}=(-1)^{\ell(\mathbf a)}q^{-|\mathbf a|/2}\sum_{(Q_{ji}:\,1\le j\le k,\,1\le i\le a_j)\atop{\deg Q_{ji}=j}}\prod_{1\le j\le k\atop{1\le i\le a_j}}\Lambda(Q_{ji})\left(\frac h{Q_{ji}}\right).$$
    Summing over $h\in\mathcal  H_{2n+1}$ and using Theorem \ref{thm: katz-sarnak} with $f(U)=\prod_{j=1}^\infty\tr(U^j)^{a_j}$ and (\ref{eq: num sqfree}) we obtain the assertion of the proposition.
    
\end{proof}

Taking $q\to\infty$, Proposition \ref{prop_explicit_formula_means_of_traces} reduces Theorem \ref{thm: main} to the next

\begin{prop}\label{prop: main}Let $\mathbf a$ be supported on $[1,k]$ (that is $a_j=0$ for $j>k$). Then
\begin{equation}\label{eq: prop main}q^{-2n-1-|\mathbf a|/2}\sum_{h\in\mathcal H_{2n+1}}\sum_{(Q_{ji}:\,1\le j\le k,\,1\le i\le a_j)\atop{\deg Q_{ji}=j}}\prod_{1\le j\le k\atop{1\le i\le a_j}}\Lambda(Q_{ji})\left(\frac h{Q_{ji}}\right)=\sum_{\mathbf b\le\mathbf a}{\mathbf a\choose\mathbf b}g(\mathbf b)\phi(n,\mathbf a-\mathbf b)+O(q^{-1/2}).\end{equation}\end{prop}

Sections \ref{sec: distinct primes}-\ref{sec: final} are dedicated to proving Proposition \ref{prop: main}. In \S\ref{sec: distinct primes} we analyze the contribution of sets of distinct primes to the LHS of (\ref{eq: prop main}), which will give the $\phi(n,\mathbf a-\mathbf b)$ term. In \S\ref{sec: squares} we analyze the contribution of prime squares and pairs of equal primes, which will give the $g(\mathbf b)$ term. Finally in \S\ref{sec: final} we put all the pieces together and establish (\ref{eq: prop main}).

\section{Multiple character sums over distinct primes}
\label{sec: distinct primes}

 Throughout this section $\mathbf a$ is supported on $[1,k]$, where $k$ is a fixed natural number. Define

\begin{equation}\label{eq: def T}
T(n;\mathbf a)=\sum_{h\,\mathrm{monic}\atop{\deg h=2n+1}}\sum_{(P_{ji}:\,1\le j\le k,\,1\le i\le a_j)\atop{\text {distinct primes}\atop{\deg P_{ji}=j}}}\prod_{1\le j\le k\atop{1\le i\le a_j}}\Lambda(P_{ji})\left(\frac h{P_{ji}}\right).
\end{equation}
These multiple character sums, initially studied in \cite{ERR13}, will be needed in order to estimate the LHS of Proposition \ref{prop: main}.

\begin{prop}\label{prop_S_sum_estimate}
Assume $|\mathbf a|\le 4n+1$. Then
\begin{equation*}
    q^{-2n-1-|\mathbf a|/2}T(n;\mathbf a)=\phi(n,\mathbf a)+O\left(q^{-1/2}\right),
\end{equation*}
where $\phi$ is defined by (\ref{eq: def phi}).

\end{prop}

\begin{proof}

We have \begin{equation}\label{eq: TS}T(n;\mathbf a)=\left(\prod_{1\le j\le k}j^{a_j}\right)S(n;\mathbf a),\end{equation} where the sums
\begin{equation}\label{eq: def S}
S(n;\mathbf a)=\sum_{h\,\mathrm{monic}\atop{\deg h=2n+1}}\sum_{(P_{ji}:\,1\le j\le k,\,1\le i\le a_j)\atop{\text {distinct primes}\atop{\deg P_{ji}=j}}}\prod_{1\le j\le k\atop{1\le i\le a_j}}\left(\frac h{P_{ji}}\right)
\end{equation}
were estimated in \cite[\S 4]{ERR13}. The notation there is slightly different: our $S(n;\mathbf a)$ is written there as $S(\beta;\vec r)$ where $\vec r$ is a vector with entries equal to $1\le j\le k$, each repeated $a_j$ times and $\beta=2n+1$. Our $\phi(n,\mathbf a)$ (in the nontrivial case $2\,\mid\,|\mathbf a|$) is denoted there by $\Phi_\beta(\vec r)$, see \cite[Equation (3.8)]{ERR13}. We divide the proof into cases depending on the value of $|\mathbf a|$. 

For $|\mathbf a|=0$ we have $T(n;\mathbf a)=q^{2n+1}$ since each term in the external sum in (\ref{eq: def T}) consists of a single empty product (which equals 1 by convention). Since $\phi(n;\mathbf a)=1$ the assertion follows in this case.

For $0<|\mathbf a|\le 2n+1$, by \cite[Lemma 4.1]{ERR13} the sum $S(n;\mathbf a)$ (and therefore $T(n;\mathbf a)$) is zero, and so is $\phi(n,\mathbf a)$ because the defining sum in (\ref{eq: def phi}) is empty. The assertion follows in this case as well.

For $|\mathbf a|=2n+2$, by \cite[equation (4.5)]{ERR13}  we have 
$$
S(n;\mathbf a)=-q^{n}\prod_{j=1}^k\pi(j)^{a_j}+O\left( q^{3n+1}\right),
$$
where $\pi(j)$ denotes the number of monic primes in $\F_q[x]$ of degree $j$
(the cited equation as written gives an error term of $O\left(q^{3n+2}\right)$, but the argument actually saves another factor of $q$ in the error term). By the Prime Polynomial Theorem we have $\pi(j)=\frac{q^j}{j}+O(q^{j/2})$. Hence (using (\ref{eq: TS}))
$$
T(n;\mathbf a)=-q^{3n+2}+O\left(q^{3n+1}\right)$$
and $$q^{-2n-1-|\mathbf a|/2}T(n;\mathbf a)=q^{-3n-2}T(n;\mathbf a)=-1+O(q^{-1})=\phi(n;\mathbf a)+O(q^{-1})$$
(using (\ref{eq: def phi}) and noting that the only summand is with $\mathbf b=\mathbf 0$, the trivial partition).

In the remaining case $|\mathbf a|\ge 2n+3$, by  \cite[Proposition 4.3]{ERR13} we have
$$
S(n;\mathbf a)=\phi(n,\mathbf a)q^{2n+1}\prod_{j=1}^k{\left(\frac{\pi(j)}{q^{j/2}}\right)^{a_j}}+O\left(q^{|\mathbf a|}\right)=\phi(n;\mathbf a)q^{2n+1-|\mathbf a|/2}\prod_{j=1}^k\pi(j)^{a_j}+O\left(q^{|\mathbf a|}\right).
$$
Once again using (\ref{eq: TS}) and  $\pi(j)=\frac{q^j}{j}+O(q^{j/2})$ we obtain
$$
q^{-2n-1-|\mathbf a|/2}T(n;\mathbf a)=\phi(n,\mathbf a)\left(1+O(q^{-1/2})\right)+O\left(q^{|\mathbf a|/2-2n-1}\right)=\phi(n,\mathbf a)+O\left(q^{-1/2}\right),
$$
using the assumption $|\mathbf a|\le 4n+1$ in the last equality. This completes the proof in all cases.
\end{proof}

\begin{lem}[Weil bound]\label{prop_weil_bound}
    Let $h\in\F_q[x]$ be a monic polynomial of degree $2n+1$, $j\in\N$. Then
    $$
    \sum_{\deg P=j}\left(\frac{h}{P}\right)=O(q^{j/2}).
    $$
\end{lem}

\begin{proof}
Write $h=h_1h_2^2$, where $h_1\neq 1$ is squarefree. Then 
$$
    \sum_{\deg P=j}\left(\frac{h}{P}\right)=
    \sum_{\substack{\deg P=j\\ P\nmid h_2}}\left(\frac{h_1}{P}\right)=\sum_{\substack{\deg P=j}}\left(\frac{h_1}{P}\right)+O(1),
$$
since the number of distinct prime divisors of $h_2$ is at most $\deg h=O(1)$. Now use the Weil explicit formula \eqref{eqn_weil_explicit_formula} with conductor $h_1$, noting that the number of prime powers of degree $j$ which are not primes is $O(q^{j/2})$. 
\end{proof}

We will need the following generalization of Lemma \ref{prop_weil_bound}.

\begin{lem}\label{lem: weil gen} Let $h\in\F_q[x]$ be a monic polynomial of degree $2n+1$, $r,j_1,\ldots,j_r$ fixed natural numbers. Then
$$\sum_{P_1,\ldots,P_r\atop{\mathrm{distinct}\atop{\deg P_l=j_l}}}\left(\frac h{P_1\cdots P_r}\right)=O\left(q^{\frac{j_1+\cdots+j_r}2}\right).$$
\end{lem}

\begin{proof}
We prove the assertion by induction on $r$, the case $r=1$ being Lemma \ref{prop_weil_bound}. Assume that $r>1$ and that the assertion holds for $r-1$. We have (below a term marked with  $\,\widehat{}\,$ is skipped)

\begin{samepage}
\begin{multline*}\sum_{P_1,\ldots,P_r\atop{\text{distinct}\atop{\deg P_l=j_l}}}\left(\frac h{P_1\cdots P_r}\right)=\sum_{P_1,\ldots,P_{r-1}\atop{\text{distinct}\atop{\deg P_l=j_l}}}\left(\frac h{P_1\cdots P_{r-1}}\right)\sum_{\deg P_r=j_r}\left(\frac h{P_r}\right)-\sum_{1\le m\le r-1\atop{j_m=j_r}}\sum_{P_1,\ldots,P_{r-1}\atop{\text{distinct}\atop{\deg P_l=j_l}}}\left(\frac h{P_1\cdots P_m^2\cdots P_{r-1}}\right)
\\=
\sum_{P_1,\ldots,P_{r-1}\atop{\text{distinct}\atop{\deg P_l=j_l}}}\left(\frac h{P_1\cdots P_{r-1}}\right)\sum_{\deg P_r=j_r}\left(\frac h{P_r}\right)-\sum_{1\le m\le r-1\atop{j_m=j_r}}\sum_{\deg P_m=j_r\atop{P_m\nmid h}}\sum_{P_1,\ldots,\hat{P}_m,\ldots,P_{r-1}\atop{\text{distinct}\atop{\deg P_l=j_l}}}\left(\frac h{P_1\cdots \widehat{P}_m\cdots P_{r-1}}\right)
\\
\ll q^{\frac{j_1+\ldots+ j_{r-1}}2}q^{\frac {j_r}2}+q^{j_r}q^{\frac{j_1+\ldots+j_{r-1}-j_r}2}\ll q^{\frac{j_1+\ldots+j_r}2},
\end{multline*}
\end{samepage}
the first inequality here being a consequence of the induction hypothesis and Lemma \ref{prop_weil_bound}.\end{proof}

In what follows we say that $Q$ is $\text{prime}^2$ (prime squared) to mean that $Q$ is the square of a prime. The following lemma allows us to restrict the sum on the LHS of Proposition \ref{prop: main} to a sum over primes and prime$^2$'s.

\begin{lem}\label{prop_not_prime_or_square_or_non_square_free_is_ngeligible} We have
    \begin{multline*}
    q^{-2n-1-|\mathbf a|/2}\sum_{h\in\mathcal{H}_{2n+1}}\sum_{(Q_{ji}:\,1\le j\le k,\,1\le i\le a_j)\atop{\deg Q_{ji}=j}}\prod_{1\le j\le k\atop{1\le i\le a_j}}\Lambda(Q_{ji})\left(\frac h{Q_{ji}}\right)=\\
    =q^{-2n-1-|\mathbf a|/2}\sum_{\deg h=2n+1\atop{\mathrm{monic}}}\sum_{(Q_{ji}:\,1\le j\le k,\,1\le i\le a_j)\atop{Q_{ji}\mathrm{\, prime\,or\, prime}^2\atop{\deg Q_{ji}=j}}}\prod_{1\le j\le k\atop{1\le i\le a_j}}\Lambda(Q_{ji})\left(\frac h{Q_{ji}}\right)+O(q^{-1}).
    \end{multline*}
\end{lem}

\begin{proof}
Let $h$ be monic of degree $2n+1$. We have 
$$
\sum_{(Q_{ji}:1\le j\le k,\,1\le i\le a_j)\atop{\deg Q_{ji}=j}}\prod_{1\le j\le k\atop{1\le i\le a_j}}\Lambda(Q_{ji})\left(\frac h{Q_{ji}}\right)=\prod_{1\le j\le k\atop{1\le i\le a_j}}\left(\sum_{\deg Q_{ji}=j}\Lambda(Q_{ji})\left(\frac{h}{Q_{ji}}\right)\right).
$$
From Lemma \ref{prop_weil_bound},
$$
\sum_{\deg Q_{ji}=j\atop{Q_{ji}\text{ not prime or prime}^2}}\Lambda(Q_{ji})\left(\frac{h}{Q_{ji}}\right)=\sum_{l=3\atop{l|j}}^\infty \sum_{Q_{ji}=f^l,\atop{f\text{ prime}}}\Lambda(Q_{ji})\left(\frac{h}{Q_{ji}}\right)=\sum_{l=3\atop{l|j}}^\infty O(q^{j/2l})=O(q^{j/2-1}).
$$
We obtain
\begin{equation}\label{eq: sig pi to pi sig with error}
\sum_{(Q_{ji}:\,1\le j\le k,\,1\le i\le a_j)\atop{\deg Q_{ji}=j}}\prod_{1\le j\le k\atop{1\le i\le a_j}}\Lambda(Q_{ji})\left(\frac h{Q_{ji}}\right)=\prod_{\substack{1\le j\le k\\ 1\le i\le a_j}}\left(\sum_{\deg Q_{ji}=j\atop{Q_{ji}\text{ prime or prime}^2}}\Lambda(Q_{ji})\left(\frac{h}{Q_{ji}}\right)
+O(q^{j/2-1})
\right).
\end{equation}
Now we note that
\begin{enumerate}
    \item[(i)] $\sum_{\substack{\deg Q_{ji}=j\\ \text{ prime}}}\Lambda(Q_{ji})\left(\frac{h}{Q_{ji}}\right)=O(q^{j/2})$. This follows from Lemma \ref{prop_weil_bound}. 
    \item[(ii)] $\sum_{\substack{\deg Q_{ji}=j\\ \text{ prime}^2}}\Lambda(Q_{ji})\left(\frac{h}{Q_{ji}}\right)=O(q^{j/2})$. This is because the number of monic squares of degree $j$ is $O(q^{j/2})$. 
\end{enumerate}
Opening up the parentheses on the RHS of (\ref{eq: sig pi to pi sig with error}) and using (i-ii) we obtain
$$
\sum_{(Q_{ji}:\,1\le j\le k,\,1\le i\le a_j)\atop{\deg Q_{ji}=j}}\prod_{1\le j\le k\atop{1\le i\le a_j}}\Lambda(Q_{ji})\left(\frac h{Q_{ji}}\right)=\sum_{(Q_{ji}:\,1\le j\le k,\,1\le i\le a_j)\atop{Q_{ji}\text{ prime or prime}^2\atop{\deg Q_{ji}=j}}}\prod_{1\le j\le k\atop{1\le i\le a_j}}\Lambda(Q_{ji})\left(\frac h{Q_{ji}}\right)+O(q^{|\mathbf a|/2-1}).
$$
Summing over all $h$ monic of degree $2n+1$ we obtain 
\begin{samepage}
\begin{multline}
\label{eq: restrict to prime and prime2}
q^{-2n-1-|\mathbf a|/2}\sum_{\deg h=2n+1\atop{\mathrm{monic}}}\sum_{(Q_{ji}:\,1\le j\le k,\,1\le i\le a_j)\atop{\deg Q_{ji}=j}}\prod_{1\le j\le k\atop{1\le i\le a_j}}\Lambda(Q_{ji})\left(\frac h{Q_{ji}}\right)
\\ =
q^{-2n-1-|\mathbf a|/2}\sum_{\deg h=2n+1\atop{\mathrm{monic}}}
\sum_{(Q_{ji}:\,1\le j\le k,\,1\le i\le a_j)\atop{Q_{ji}\text{ prime or prime}^2\atop{\deg Q_{ji}=j}}}\prod_{1\le j\le k\atop{1\le i\le a_j}}\Lambda(Q_{ji})\left(\frac h{Q_{ji}}\right)+O(q^{-1}).
\end{multline}  
\end{samepage}

To complete the proof of the lemma we need to show that the LHS of (\ref{eq: restrict to prime and prime2}) changes by $O(q^{-1})$ if we restrict the summation to squarefree $h$. We note that the number of non-squarefree $h$ is $q^{2n}$ (by (\ref{eq: num sqfree})) and for a given $h$,    
    $$\sum_{(Q_{ji}:1\le j\le k,\,1\le i\le a_j)\atop{\deg Q_{ji}=j}}\prod_{1\le j\le k\atop{1\le i\le a_j}}\Lambda(Q_{ji})\left(\frac h{Q_{ji}}\right)=O(q^{|\mathbf a|/2})$$
by Lemma \ref{prop_weil_bound}, so the contribution of non-squarefree $h$ is $O(q^{-1})$.
\end{proof}


\section{The contribution of squares}
\label{sec: squares}

We introduce some notational conventions that will be used in \S\ref{sec: squares}-\ref{sec: final}.

\begin{notation}\label{not: ji}\begin{enumerate}\item[(i)] In what follows the notation $ji$ (and similar) is a shorthand for the pair $(j,i)$.
\item[(ii)] In what follows by $Q_{ji}$ we always denote a monic prime or prime$^2$ in $\F_q[x]$ with $\deg Q_{ji}=j$. In particular summation over collections $(Q_{ji})$ will always mean summation over all suitable collections with $Q_{ji}$ satisfying the above requirements.
\end{enumerate}\end{notation}

\begin{definition}\label{def: pairing}
Let $\mathbf b=(b_j)_{j=1}^\infty\in\mathcal P$ be a partition.
A  \emph{pairing} of $\mathbf b$ is a permutation $\sigma$ of (the finite set) $\{ji:1\le j< \infty,1\le i\le b_j\}$ such that $\sigma$ preserves the $j$-coordinate, $\sigma^2=\mathrm{id}$ and $\sigma$ has no fixed point $ji$ with odd $j$.

\end{definition}

Note that the set of indices appearing in the above definition enumerates the parts of $\mathbf b$ by taking the $(ji)$-th part to be $j$. A {pairing} can be viewed as a way to partially divide the parts of $\mathbf b$ into pairs of equal sizes, such that no odd part is left unpaired ($\sigma$ transposes each pair and fixes the unpaired parts).

\[
\begin{tikzpicture}[x=0.55cm,y=0.55cm,>=stealth]

\foreach \c/\height in {
  1/4, 2/4, 3/4,
  4/3, 5/3, 6/3, 7/3,
  8/2,
  9/1, 10/1
}{
  \foreach \r in {1,...,\height}{
    \draw (\c-1,\r-1) rectangle (\c,\r);
  }
}

\draw[<->]
  (0.5,0)
  .. controls (0.5,-1.0) and (2.5,-1.0) ..
  (2.5,0);

\draw[<->]
  (3.5,0)
  .. controls (3.5,-1.0) and (5.5,-1.0) ..
  (5.5,0);

\draw[<->]
  (4.5,0)
  .. controls (4.5,-1.7) and (6.5,-1.7) ..
  (6.5,0);

\draw[<->]
  (8.5,0)
  .. controls (8.5,-0.75) and (9.5,-0.75) ..
  (9.5,0);

\end{tikzpicture}
\]

\begin{lem}\label{prop_number_of_pairings}
Let $\mathbf b$ be a partition. For a pairing $\sigma$ of $\mathbf b$ denote by $b^{(\sigma)}_j=\frac 12|\{1\le i\le b_j:\sigma(ji)\neq ji\}|$ the number of pairs of $j$-parts induced by $\sigma$. Then \begin{equation}\label{eq: sum sigma}\sum_{\sigma}\prod_{j=1}^\infty j^{b^{(\sigma)}_j}=g(\mathbf b),\end{equation}
where the sum is over all pairings of $\mathbf b$ and $g(\mathbf b)$ is defined by (\ref{eq: def g}).
\end{lem}

\begin{proof}
By the structure of pairings (namely that they preserve the $j$-coordinate) and the multiplicative definition of $g(\mathbf b)$, it suffices to prove the claim for the case of $\mathbf b=(0,\ldots,0, b,0,\ldots)$ supported on a single $j\in\N$. Now we note that the number of ways to divide a set of size $b$ into (unordered) pairs is $\frac{b!}{2^{b/2}(b/2)!}=(b-1)!!$ if $2|b$ and 0 otherwise. This proves the result for the case of $2\nmid j$ (because in this case all the parts of $\mathbf b$ must be paired and $b_j^{(\sigma)}=b/2$). 

In the case $2|j$, we need to choose a subset of even size $2l$ of $\{i:1\le i\le b\}$ and divide this subset into pairs, which can be done in $\sum_{l=0}^\infty {b\choose 2l}(2l-1)!!$ ways and gives a total contribution of $\sum_{l=0}^\infty {b\choose 2l}(2l-1)!!j^l=g_j(b)=g(\mathbf b)$ to (\ref{eq: sum sigma}) (here $b_j^{(\sigma)}=l$). This completes the proof.
\end{proof}

\begin{prop}\label{prop_sum_to_pair_up}
Let $\mathbf b$ be a partition supported on $[1,k]$. Then,

\begin{equation}\label{eq: contribution of squares} q^{-|\mathbf b|/2}\sum_{(Q_{ji}:\,1\le j\le k,\,1\le i\le b_j)\atop{\prod Q_{ji}=\square}}\prod_{1\le j\le k\atop{1\le i\le b_j}}\Lambda(Q_{ji})=g(\mathbf b)+O(q^{-1}),
\end{equation}
where $\prod Q_{ji}$ denotes the product over $1\le j\le k,1\le i\le b_j$.
\end{prop}

\begin{proof}
First note that if for some odd $j_1$ we have $2\nmid b_{j_1}$, then no suitable choice of the $Q_{ji}$ can make $\prod Q_{ji}$ a square (because the $Q_{j_1i}$ must be prime and the total multiplicity of primes of degree $j_1$ in the product is odd), in agreement with $g(\mathbf b)=0$ in this case. Thus we may assume that for every odd $j$ we have $2|b_j$. 

We say that a choice $(Q_{ji}:1\le j\le k,1\le i\le b_j)$ respects a pairing $\sigma$ of $\mathbf b$ if
\begin{enumerate}
    \item[(i)] Whenever $\sigma(ji_1)=ji_2\neq ji_1$ are paired we have $Q_{ji_1}=Q_{ji_2}=P$ for $P$ prime.
    \item[(ii)] If $\sigma(ji_1)=ji_1$ is not paired (by our definition of a pairing $j$ is even in this case), then $Q_{ji_1}=P^2$ for $P$ prime of degree $j/2$.
\end{enumerate}
If $(Q_{ji})$  respects a pairing $\sigma$ then $\prod Q_{ji}=\square$. Conversely it is easy to see that every suitable choice of $Q_{ji}$ with $\prod Q_{ji}=\square$ respects some (not necessarily unique) pairing $\sigma$ of $\mathbf b$. We denote the contribution of the pairing $\sigma$ to (\ref{eq: contribution of squares}) by
$$
S_{\sigma}=q^{-|\mathbf b|/2}\sum_{(Q_{ji}:\,1\le j\le k,\,1\le i\le b_j)\atop{Q\text{ respects }\sigma}}\prod_{1\le j\le k\atop{1\le i\le b_j}}\Lambda(Q_{ji}).
$$

Let $\sigma\neq\sigma'$ be two different pairings of $\mathbf b$. The sums $S_{\sigma},S_{\sigma'}$ might share terms, but the contribution of the shared terms is negligible. Indeed, if a choice of $(Q_{ji})$ respects both $\sigma$ and $\sigma'$, then there must be a prime $P$ such that $P^4|\prod Q_{ji}$ (otherwise the pairing $\sigma$ is determined uniquely by which of the prime values of $Q_{ji}$ coincide). The number of possible values of $\prod Q_{ji}$ is therefore $O(q^{|\mathbf b|/2-1})$ (since its square root is not squarefree) and since each of them has $O(1)$ possible divisors, the contribution of such terms to $S_{\sigma}$ is $O(q^{-1})$. Noting that the number of pairings is also $O(1)$, we obtain
\begin{equation}\label{eq: sum of S sigma}q^{-|\mathbf b|/2}\sum_{(Q_{ji}:\,1\le j\le k,\,1\le i\le b_j)\atop{\prod Q_{ji}=\square}}\prod_{1\le j\le k\atop{1\le i\le b_j}}\Lambda(Q_{ji})=\sum_{\sigma\text{ pairing of }\mathbf b}S_\sigma +O(q^{-1}).
\end{equation}

Now for a pairing $\sigma$ of $\mathbf b$ we have (using $\pi(j)=\frac {q^j}j+O(q^{j-1})$)
\begin{multline*}
S_{\sigma}=\left(\prod_{1\le j\le k\atop{1\le i\le b_j\atop{\sigma(ji)= ji',\,i'>i}}}q^{-j}\sum_{\deg P=j\atop{\text{ prime}}}\Lambda(P)^2\right)\cdot\left(\prod_{{1\le j\le k\atop{1\le i\le b_j\atop{\sigma(ji)=ji}}}}q^{-j/2}\sum_{\deg P=j/2\atop{\text{prime}}}\Lambda(P)\right)=\\
=\prod_{1\le j\le k\atop{1\le i\le b_j\atop{\sigma(ji)= ji',\,i'>i}}}\left(j+O(q^{-1})\right)\prod_{{1\le j\le k\atop{1\le i\le b_j\atop{\sigma(ji)=ji}}}}\left(1+O(q^{-1})\right)=\prod_{j=1}^k j^{b_j^{(\sigma)}}+O(q^{-1}),
\end{multline*}
where $b_j^{(\sigma)}$ was defined in Lemma \ref{prop_number_of_pairings}. Summing over all pairings $\sigma$ of $\mathbf b$ and using (\ref{eq: sum of S sigma}) and Lemma \ref{prop_number_of_pairings} we obtain the assertion (\ref{eq: contribution of squares}).

\end{proof}

\section{Final assembly}
\label{sec: final}

In the present section we combine the estimates of Proposition \ref{prop_S_sum_estimate} and Proposition \ref{prop_sum_to_pair_up} to derive Proposition \ref{prop: main}, which as we have seen in \S\ref{sec: hyperelliptic} completes the proof of Theorem \ref{thm: main}. We maintain Notation \ref{not: ji} throughout this section, in particular $Q_{ji}$ is always a prime or prime$^2$ with $\deg Q_{ji}=j$.

Throughout this section $\mathbf a$ is supported on $[1,k]$, where $k$ is a fixed natural number.
Denote \begin{equation}\label{eq: def A}A=\{ji:\,1\le j\le k,\,1\le i\le a_j\}.\end{equation}
\begin{definition}\label{def: compatible}Let $K\subset A$ be a subset and $\sim$ an equivalence relation on $A\setminus K$. We say that a collection $(Q_{ji}:\,1\le j\le k,\,1\le i\le a_j)$ is \emph{compatible with} $\sim$ if the following conditions hold:
\begin{enumerate}\item[(i)] $Q_{ji}=\text{prime}^2$ iff $ji\in K$ (so $Q_{ji}$ is prime iff $ji\in A\setminus K$).
\item[(ii)] If $ji,j'i'\in A\setminus K$ then $Q_{ji}=Q_{j'i'}$ iff $ji\sim j'i'$ (in particular $j=j'$ if $ji\sim j'i'$).
\end{enumerate}
\end{definition}

\begin{lem}\label{lem_sums_in_inclusion_exclusion} Let $K\subset A$ be a subset and $\sim$ an equivalence relation on $A\setminus K$ such that at least one equivalence class has size $\ge 3$. Then
\begin{equation}\label{eq: lem triple}q^{-2n-1-|\mathbf a|/2}\sum_{\deg h=2n+1\atop{\text{monic}}}\sum_{(Q_{ji}:\,1\le j\le k,\,1\le i\le a_j)\atop{\text{compatible with $\sim$}}}\prod_{1\le j\le k\atop{1\le i\le a_j}}\Lambda(Q_{ji})\left(\frac h{Q_{ji}}\right)=O(q^{-1}).\end{equation}
\end{lem}

\begin{proof} Denote by $\mathcal C$ the set of equivalence classes of $\sim$. For any $\alpha=ji\in A$ denote $j(\alpha)=j$ and for $C\in\mathcal C$ let $j(C)=j(\alpha)$ for some (and therefore any) $\alpha\in C$. Denote
$$\mathcal C_0=\{C\in\mathcal C:\,2\,\mid\,|C|\},\quad \mathcal C_1=\{C\in\mathcal C:\,2\,\nmid\,|C|\}.$$ The LHS of (\ref{eq: lem triple}) can be rewritten (identifying the $Q_{ji}$ that must be equal due to compatibility with $\sim$) as
\begin{multline*}q^{-2n-1-|\mathbf a|/2}\sum_{\deg h=2n+1\atop{\text{monic}}}\prod_{\kappa\in K}\left(\sum_{Q=\text{prime}^2\atop{\deg Q=j(\kappa)}}\Lambda(Q)\left(\frac h{Q}\right)\right)\sum_{(P_C:\,C\in\mathcal C)\atop{\text{distinct primes}\atop{\deg P_C=j(C)}}}\prod_{C\in\mathcal C}\Lambda(P_C)^{|C|}\left(\frac h{P_C}\right)^{|C|}
\\
=q^{-2n-1-|\mathbf a|/2}\sum_{\deg h=2n+1\atop{\text{monic}}}\prod_{\kappa\in K}\left(\sum_{Q=\text{prime}^2\atop{\deg Q=j(\kappa)}}\Lambda(Q)\left(\frac h{Q}\right)\right)
\\
\cdot\sum_{(P_C:\,C\in\mathcal C_0)\atop{\text{distinct primes}\atop{\deg P_C=j(C)\atop{P_C\nmid h}}}}\left(\prod_{C\in\mathcal C_0}j(C)^{|C|}\prod_{D\in\mathcal C_1}j(D)^{|D|}
\sum_{(P_D:\,D\in\mathcal C_1)\atop{\text{distinct primes}\atop{\deg P_D=j(D)}}}\left(\frac{h{\prod_{C\in\mathcal C_0}}P_C^2}{P_D}\right)\right).
\end{multline*}

Applying Lemma \ref{lem: weil gen} to the innermost sum, the above expression is
\begin{equation}\label{eq: bound1}\ll q^{-2n-1-|\mathbf a|/2}q^{2n+1}q^{\frac 12\sum_{\kappa\in K}j(K)}q^{\sum_{C\in\mathcal C_0}j(C)}q^{\frac 12\sum_{D\in\mathcal C_1}j(D)}\le q^{-1},\end{equation}
the latter inequality is because by the assumption of the lemma we have either $|C|\ge 4$ for some $C\in\mathcal C_0$ (note that $|C|\ge 2$ for all $C\in\mathcal C_0$) or $|D|\ge 3$ for some $D\in\mathcal C_1$ and therefore $$\frac 12\sum_{\kappa\in K}j(K)+\sum_{C\in\mathcal C_0}j(C)+\frac 12\sum_{D\in\mathcal C_1}j(D)\le\frac 12\sum_{j=1}^kja_j-1 =\frac{|\mathbf a|}2-1$$
($j(C)$ is counted $|C|$ times in the sum on the RHS but only once on the LHS, same for $j(D)$). The bound (\ref{eq: bound1}) completes the proof.

\end{proof}

\begin{prop}\label{prop_main_sum_estimate}
    Let $\mathbf a$ be supported on $[1,k]$, and assume that $|\mathbf a|\le 4n+1$. Then
    \begin{equation*}
    R:=q^{-2n-1-|\mathbf a|/2}\sum_{\deg h=2n+1\atop{\mathrm{monic}}}\sum_{(Q_{ji}:\,1\le j\le k,\,1\le i\le a_j)}\prod_{1\le j\le k\atop{1\le i\le a_j}}\Lambda(Q_{ji})\left(\frac{h}{Q_{ji}}\right)
    =\sum_{\mathbf b\le\mathbf a}{\mathbf a\choose\mathbf b}g(\mathbf b)\phi(n,\mathbf a-\mathbf b)+O(q^{-{1/2}}).
    \end{equation*}
\end{prop}

\begin{proof} In what follows $\subset$ denotes non-strict inclusion of sets. We split the sum on the LHS of the proposition into terms in the following way: for every choice of subset $J\subset A$ of the indices (recall that $A$ is defined by (\ref{eq: def A})), we say that a choice $(Q_{ji}:1\le j\le k,1\le i\le a_j)$ respects $J$ if
\begin{enumerate}
    \item[(i)] $\prod_{ji\in J}Q_{ji}=\square$,
    \item[(ii)] $Q_{ji}$ for $ji\notin J$ are distinct primes.
\end{enumerate} We now consider the sums 
$$
R_J=q^{-2n-1-|\mathbf a|/2}
\sum_{\deg h=2n+1\atop{\mathrm{monic}}}
\sum_
{(Q_{ji}:\,1\le j\le k,\,1\le i\le a_j)\atop{\text{respects } J}
}
\prod_{1\le j\le k\atop{1\le i\le a_j}}\Lambda(Q_{ji})\left(\frac{h}{Q_{ji}}\right).
$$
We also define for subsets $J_1,\ldots,J_s\subset A$,
$$
R_{J_1,\ldots,J_s}=q^{-2n-1-|\mathbf a|/2}
\sum_{\deg h=2n+1\atop{\mathrm{monic}}}
\sum_
{(Q_{ji}:\,1\le j\le k,\,1\le i\le a_j)\atop{\text{respects } J_1,\ldots,J_s}
}
\prod_{1\le j\le k\atop{1\le i\le a_j}}\Lambda(Q_{ji})\left(\frac{h}{Q_{ji}}\right).
$$
Since any choice of $Q_{ji}$ respects some subset $J$ (collect all the squares and pairs of repeating primes so that only distinct primes remain), by inclusion-exclusion we have 
\begin{equation}\label{eq: inclusion-exclusion}
R =\sum_{J_1}R_{J_1}-\sum_{J_1,J_2}R_{J_1,J_2}+\sum_{J_1,J_2,J_3}R_{J_1,J_2,J_3}-\ldots,
\end{equation}
where summation in each term is over unordered tuples of distinct subsets of $A$.
We will show that $R_{J_1,\ldots,J_s}=O(q^{-1})$ for $s\ge 2$ (and $J_1,\ldots,J_s$ distinct) and evaluate $R_J$ for every $J$, which will complete the evaluation of $R$ (up to an $O(q^{-1})$ error term).

Let $J_1,\ldots,J_s,\,s\ge 2$ be distinct. Consider a collection $(Q_{ji}:\,1\le j\le k,\,1\le i\le a_j)$. Denote $K=\{ji: Q_{ji}=\text{prime}^2\}$ and define an equivalence relation $\sim$ on $A\setminus K$ (which consists of the prime $Q_{ji}$) by $ji\sim j'i'$ iff $Q_{ji}=Q_{j'i'}$. This is the unique choice of $(K,\sim)$ that is compatible with $(Q_{ji})$ in the sense of Definition \ref{def: compatible}. The collection $(Q_{ji})$ respects $J_l$ iff the following conditions hold:
\begin{enumerate}
\item[(i)] $K\subset J_l$.
\item[(ii)] For each equivalence class $C$ of $\sim$ we have that $|C\cap J_l|$ is even and $|C\setminus J_l|\le 1$.
\end{enumerate}
Therefore the condition that $(Q_{ji})$ respects $J_1,\ldots,J_s$ depends only on $(K,\sim)$. Moreover if (i-ii) above hold with two distinct $J_1,J_2$ then one of the equivalence classes $C$ must be of size $\ge 3$. To see this let (WLOG) $uv\in J_1\setminus J_2$. Then $Q_{uv}$ is a prime which appears at least twice in $(Q_{ji})_{ji\in J_1}$ (because $\prod_{ji\in J_1}Q_{ji}=\square$) and therefore at least once in $(Q_{ji})_{ji\in J_2}$ (because $(Q_{ji})_{ji\in A\setminus J_2}$ are distinct primes). It follows that $Q_{uv}$ appears at least twice in $(Q_{ji})_{ji\in J_2}$ (because $\prod_{ji\in J_2}Q_{ji}=\square$) and since $uv\not\in J_2$ it appears $\ge 3$ times in $(Q_{ji})_{ji\in A}$. Its equivalence class $C$ has therefore size $\ge 3$. 

Using Lemma \ref{lem_sums_in_inclusion_exclusion} we obtain
$$R_{J_1,\ldots,J_s}=\sum_{K,\sim}q^{-2n-1-|\mathbf a|/2}\sum_{\deg h=2n+1\atop{\text{monic}}}\sum_{(Q_{ji}:\,1\le j\le k,\,1\le i\le a_j)\atop{\text{compatible with $\sim$}}}\prod_{1\le j\le k\atop{1\le i\le a_j}}\Lambda(Q_{ji})\left(\frac h{Q_{ji}}\right)=O(q^{-1}).$$
Here the sum $\sum_{K,\sim}$ is over all $(K,\sim)$ such that compatibility with it implies respecting $J_1,\ldots,J_s$.

We conclude that $R_{J_1,\ldots,J_s}=O(q^{-1})$ and from (\ref{eq: inclusion-exclusion}) that
\begin{equation}\label{eq: R through RJ}R=\sum_{J\subset A}R_J+O(q^{-1}).\end{equation}
Next we evaluate $R_J$ for $J\subset A$.
With each $J\subset A$ we associate a partition $\mathbf b\le\mathbf a$ by $b_j=|\{1\le i\le a_j: ji\in J\}|$.
The number of subsets $J$ corresponding to a given $\mathbf b\le\mathbf a$ is exactly $\mathbf a\choose\mathbf b$ and $R_J$ depends only on $\mathbf b$ and not on the specific $J$. Hence from (\ref{eq: R through RJ}) we have
\begin{equation}\label{eq: R through Rb}R=\sum_{\mathbf b\le\mathbf a}{\mathbf a\choose\mathbf b}R_{\mathbf b}+O(q^{-1}),\end{equation}
where
$$R_{\mathbf b}=q^{-2n-1-|\mathbf a|/2}
\sum_{\deg h=2n+1\atop{\mathrm{monic}}}
\sum_
{(Q_{ji}:\,1\le j\le k,\,1\le i\le b_j)\atop{\prod_{1\le j\le k\atop{1\le i\le b_j}}Q_{ji}=\square\atop{(Q_{ji})_{1\le j\le k\atop{b_j<i\le a_j}}\text{distinct primes}}}}
\prod_{1\le j\le k\atop{1\le i\le a_j}}\Lambda(Q_{ji})\left(\frac{h}{Q_{ji}}\right)
$$
(the latter expression is just $R_J$ for $J=\{ji:\,1\le j\le k,\,1\le i\le b_j\}$, which corresponds to $\mathbf b$).
It remains to evaluate $R_{\mathbf b}$ for each $\mathbf b\le\mathbf a$.

Fix a partition $\mathbf b\le \mathbf a$. Let $(Q_{ji}:\,1\le j\le k,\,1\le i\le a_j)$ be a choice respecting $J=\{ji:\,1\le j\le k,\,1\le i\le b_j\}.$ Recall that this means that $\prod_{1\le j\le k\atop{1\le i\le b_j}}Q_{ji}=\square$ and $Q_{ji},\,1\le j\le k,\,b_j< i\le a_j$ are distinct primes. If we fix $Q_{ji}$ for $1\le j\le k,\,1\le i\le b_j$ then 
\begin{multline*}
\sum_{\deg h=2n+1\atop{\text{monic}}}\sum_{(Q_{ji}:\,1\le j\le k,\,b_j<i\le a_j)\atop{\text{distinct primes}}}\prod_{1\le j\le k\atop{1\le i\le a_j}}\Lambda(Q_{ji})\left(\frac{h}{Q_{ji}}\right)=
\\=\prod_{1\le j\le k\atop{1\le i\le b_j}}\Lambda(Q_{ji})\sum_{\deg h=2n+1\atop{\text{monic}\atop{\left(h,\prod_{1\le j\le k,\,1\le i\le b_j} Q_{ji}\right)=1}}}\sum_{(Q_{ji}:\,1\le j\le k,\,b_j<i\le a_j)\atop{\text{distinct primes}}}
\prod_{1\le j\le k\atop{b_j<i\le a_j}}\Lambda(Q_{ji})\left(\frac{h}{Q_{ji}}\right)=\\
=\prod_{1\le j\le k\atop{1\le i\le b_j}}\Lambda(Q_{ji})\sum_{\deg h=2n+1\atop{\text{monic}}}\sum_{(Q_{ji}:\,1\le j\le k,\, b_j<i\le a_j)\atop{\text{distinct primes}}}\prod_{1\le j\le k\atop{b_j<i\le a_j}}\Lambda(Q_{ji})\left(\frac{h}{Q_{ji}}\right)+O(q^{2n+|\mathbf a-\mathbf b|/2}).
\end{multline*}
When $\mathbf b$ is not the empty partition, the second equality holds because there are $O(q^{2n})$ polynomials $h$ which are not coprime with $\prod_{1\le j\le k\atop{1\le i\le b_j}}Q_{ji}$ and by Lemma \ref{lem: weil gen} each of them contributes $O(q^{|\mathbf a-\mathbf b|/2})$ to the sum. When $\mathbf b$ is the empty partition, the transition is immediate. Summing over the $O(q^{|\mathbf b|/2})$ appropriate choices of $(Q_{ji})_{1\le j\le k,\,1\le i\le b_j}$ (i.e. such that the product is a square) and using the notation (\ref{eq: def T}), Proposition \ref{prop_S_sum_estimate} and Proposition \ref{prop_sum_to_pair_up} (the propositions applied to $\mathbf a-\mathbf b$ and $\mathbf b$ respectively) we obtain
\begin{align*}
R_{\mathbf b}&=q^{-2n-1-|\mathbf a|/2}\left(\sum_{(Q_{ji}:\,1\le j\le k,\,1\le i\le b_j)\atop\prod Q_{ji}=\square}\prod_{1\le j\le k\atop{1\le i\le b_j}}\Lambda(Q_{ji})\right)\left(\sum_{\deg h=2n+1\atop{\text{monic}}}\sum_{(Q_{ji}:\,1\le j\le k,\,b_j<i\le a_j)\atop{\text{distinct primes}}}\prod_{1\le j\le k\atop{b_j<i\le a_j}}\Lambda(Q_{ji})\left(\frac{h}{Q_{ji}}\right)\right)\\&+O(q^{-1})
\\&=
q^{-2n-1-|\mathbf a|/2}\left(\sum_{(Q_{ji}:\,1\le j\le k,\,1\le i\le b_j)\atop\prod {Q_{ji}}=\square}\prod_{1\le j\le k\atop{1\le i\le b_j}}\Lambda(Q_{ji})\right)T(\mathbf a-\mathbf b)+O(q^{-1})
\\&=g(\mathbf b)\phi(n,\mathbf a-\mathbf b)+O(q^{-1/2}).\end{align*}
Combining this with (\ref{eq: R through Rb}) gives the assertion of Proposition \ref{prop_main_sum_estimate}.\end{proof}
\begin{proof}[Proof of Proposition \ref{prop: main}]
The assertion follows by combining Lemma \ref{prop_not_prime_or_square_or_non_square_free_is_ngeligible} and Proposition \ref{prop_main_sum_estimate}.
\end{proof}

\begin{proof}[Proof of Theorem \ref{thm: main}] As observed in \S\ref{sec: hyperelliptic}, Theorem \ref{thm: main} follows from Proposition \ref{prop: main}.\end{proof}

\section{Application to linear statistics of eigenvalues}
\label{sec: waxman}

In the present section we prove Theorem \ref{thm: waxman}.
For the rest of the section we assume that $f\in C^{\infty}(\R/\Z)$ is a fixed even real-valued function, $m$ is a fixed natural number, $\epsilon>0$ is a fixed constant. All asymptotic notation pertains to the limit $n\to\infty$ and has implicit constants and rate of convergence which may depend on $f,m,\epsilon$.

\begin{lem}\label{lem: waxman} Assume $\nu\ge\epsilon n$ and let $\mathbf a\in\mathcal P$ be such that $|\mathbf a|\le 4n+1$ and $a_j=0$ whenever $|j-\nu|>n^{1/2}$. Then $$M(\USp(2n),\mathbf a)=\left(\prod_{j=1}^\infty \eta_{a_j}(a_j-1)!!\right)\nu^{\ell(\mathbf a)/2}+O\left(n^{(\ell(\mathrm a)-1)/2}\right).$$\end{lem}

\begin{proof}
It follows from (\ref{eq: def g}) and the assumption on $a_j$ that only $j=\nu+O(n^{1/2})$ contribute to the RHS of (\ref{eq: thm1}). By (\ref{eq: def g}), for such a $j$ we have $g_j(\mathrm a_j)=\eta_{a_j}(a_j-1)!!\nu^{a_j/2}+O\left(n^{(a_j-1)/2}\right))$ and therefore $$g(\mathbf a)=\left(\prod_{j=1}^\infty \eta_{a_j}(a_j-1)!!\right)\nu^{\ell(\mathbf a)/2}+O\left(n^{(\ell(\mathbf a)-1)/2}\right)$$ and $g(\mathbf b)=O(n^{(\ell(\mathbf a)-1)/2})$ for any $\mathbf b<\mathbf a$. Since $\nu\ge\epsilon n$ and $a_j\neq 0$ only if $j=\nu+O(n^{1/2})$, we have $\ell(\mathbf a)=O(1)$ and therefore $\phi(n;\mathbf a-\mathbf b)=O(1)$ and the number of possible $\mathbf b$ (\ref{eq: thm1}) is also $O(1)$. The assertion follows from (\ref{eq: thm1}) and the observations above.
\end{proof}

\begin{proof}[Proof of Theorem \ref{thm: waxman}] Consider the Fourier series $f(t)=\sum_{j\in\Z}\hat f(j)e(jt).$ Since $f$ is even and real valued the same is true for $\hat f$. Since $f\in C^\infty(\R/\Z)$ the coefficients $\hat f(j)$ decay faster than any fixed power of $j$ as $|j|\to\infty$. We have

\begin{multline*}W_{f,\nu}(U)^m=\left[\sum_{k=1}^{2n}\sum_{j\in\Z}\hat f(j)e((\nu+j)\theta_k)\right]^m=
\sum_{j_1,\ldots,j_m\in\Z}\hat f(j_1)\cdots\hat f(j_m)\prod_{i=1}^m\mathrm{tr}(U^{j_i+\nu})
\\=\sum_{j_1,\ldots,j_m\in\Z}\hat f(j_1-\nu)\cdots\hat f(j_m-\nu)\prod_{i=1}^m\mathrm{tr}(U^{j_i}).\end{multline*}

Using Lemma \ref{lem: waxman}, the condition $\epsilon n\le\nu\le\frac{4n}m-n^{1/2}$ and the fast decay of $\hat f$ we calculate:
\begin{align*}
\int_{\USp_{2n}}W_{f,\nu}(U)^m\mathrm dU&=\sum_{j_1,\ldots,j_m\in\Z}\hat f(j_1-\nu)\cdots\hat f(j_m-\nu)\int_{\USp(2n)}\prod_{i=1}^m\mathrm{tr}(U^{j_i})\mathrm dU
\\&=
\sum_{j_1,\ldots,j_m\in\N\atop{|j_i-\nu|<n^{1/2}}}\hat f(j_1-\nu)\cdots\hat f(j_m-\nu)\int_{\USp(2n)}\prod_{i=1}^m\mathrm{tr}(U^{j_i})\mathrm dU+O(1)
\\&=\sum_{\mathbf a\in\mathcal P\atop{\ell(\mathbf a)=m\atop{a_j\neq 0\Rightarrow|j-\nu|<n^{1/2}}}}\frac{m!}{\prod_{j=1}^\infty a_j!}\prod_{j=1}^\infty \hat f(j-\nu)^{a_j}M(\USp(2n),\mathbf a)+O(1)
\\&=\sum_{\mathbf a\in\mathcal P\atop{\ell(\mathbf a)=m\atop{2|a_j\forall j\atop{a_j\neq 0\Rightarrow|j-\nu|<n^{1/2}}}}}\frac{m!}{\prod_{j=1}^\infty a_j!}\nu^{m/2}\prod_{j=1}^\infty  (a_j-1)!!\hat f(j-\nu)^{a_j}+O(n^{(m-1)/2})
\\&=\sum_{\mathbf a\in\mathcal P\atop{\ell(\mathbf a)=m\atop{2|a_j\forall j\atop{a_j\neq 0\Rightarrow|j-\nu|<n^{1/2}}}}}\frac{m!}{\prod_{j=1}^\infty 2^{a_j/2}(a_j/2)!}\nu^{m/2}\prod_{j=1}^\infty  \hat f(j-\nu)^{a_j}+O(n^{(m-1)/2})
\\&=\eta_m\frac{m!\nu^{m/2}}{2^{m/2}(m/2)!}\sum_{\mathbf a\in\mathcal P\atop{\ell(\mathbf a)=m\atop{2|a_j\forall j\atop{a_j\neq 0\Rightarrow|j-\nu|<n^{1/2}}}}}\frac{(m/2)!}{\prod_{j=1}^\infty (a_j/2)!}\prod_{j=1}^\infty  \hat f(j-\nu)^{a_j}+O(n^{(m-1)/2})
\\&=\eta_m\frac{m!\nu^{m/2}}{2^{m/2}(m/2)!}\sum_{\mathbf a\in\mathcal P\atop{\ell(\mathbf a)=m\atop{2|a_j\forall j}}}\frac{(m/2)!}{\prod_{j=1}^\infty (a_j/2)!}\prod_{j=1}^\infty  \hat f(j-\nu)^{a_j}+O(n^{(m-1)/2})
\\&=\eta_m(m-1)!!\cdot\nu^{m/2}\left(\sum_{j\in\Z}^\infty\hat f(j)^2\right)^{m/2}+O(n^{(m-1)/2})\\&=\eta_m(m-1)!!\cdot\nu^{m/2
}\|f\|_{L^2}^{m/2}+O(n^{(m-1)/2}).
\end{align*}
\end{proof}

\begin{rem} The non-Gaussian corrections in (\ref{eq: thm1}) make a negligible contribution to the calculation above, which explains why the final result is Gaussian. Nevertheless the full range of Theorem \ref{thm: main} is needed to ensure that the relevant error terms are negligible.\end{rem}

\noindent{\bf Acknowledgments.} We would like to thank Ze\'ev Rudnick for suggesting the application to linear statistics of eigenvalues. We would also like to thank the two anonymous referees of the paper for multiple minor corrections and suggestions for improving the exposition.
\\ \\
{\bf Funding.} The first author was partially supported by Israel Science Foundation grant no. 2507/19.
\\ \\
{\bf Data availability statement.} Data sharing not applicable to this article as no datasets were generated or
analyzed during the current study.

\bibliography{mybib}
\bibliographystyle{alpha}

\end{document}